\documentclass[11pt]{amsart}
\usepackage{amsmath,amsfonts,amssymb,mathabx}
\usepackage{amsthm}
\usepackage{amstext}
\usepackage[normalem]{ulem}
\usepackage{verbatim}
\usepackage{longtable}
\usepackage{enumitem}
\usepackage{hyperref}
\hypersetup{
    colorlinks,
    linkcolor={blue},
    citecolor={red},
    urlcolor={blue}
}
\usepackage{cancel}
\makeatletter
\newcommand{\setword}[2]{%
  \phantomsection
  #1\def\@currentlabel{\unexpanded{#1}}\label{#2}%  
}
\makeatother
\usepackage[usenames,dvipsnames]{color}

\scrollmode

\newtheorem{thm}{Theorem}[section]
\newtheorem{lem}[thm]{Lemma}

\newtheorem{rem}[thm]{Remark}

\newtheorem{prop}[thm]{Proposition}
\newtheorem{df}[thm]{Definition}

\newtheorem{ex}[thm]{Example}

\newtheorem{symbs}[thm]{Notations}

\newcommand{\wh}{\widehat}
\newcommand{\wt}{\widetilde}

\newcommand{\vS}{\varSigma}

\newcommand{\vO}{\varOmega}
\newcommand{\vT}{\varTheta}
\newcommand{\ov}{\overline}
\newcommand{\mf}{\mathfrak}
\newcommand{\vY}{\varUpsilon}

\newcommand{\uph}{\upharpoonright}

\def\N{{\mathbb N}}
\def\R{{\mathbb R}}
\def\F{{\mathcal F}}

\def\B{{\mathfrak B}}

\def\L{{\mathcal L}}
\def\M{{\mathcal M}}

\newcommand{\E}{\mathbb{E}}

\newcommand{\ga}{\gamma}

\newcommand{\ttheta}{\rho(\theta)}
\newcommand{\Ttheta}{\rho(\vT)}

\newcommand{\Nt}{\{N_{t}\}_{t\in\mathbb R_{+}}}
\newcommand{\Wn}{\{W_{n}\}_{n\in\mathbb N}}
\newcommand{\Tn}{\{T_{n}\}_{n\in\mathbb N_0}}
\newcommand{\Xn}{\{X_{n}\}_{n\in\mathbb N}}
\newcommand{\St}{\{S_{t}\}_{t\in\mathbb R_{+}}}

\newcommand{\Qtd}{\{Q_{\theta}\}_{\theta\in D}}

\title{Applications of a change of measures technique for compound mixed renewal processes to the ruin problem}
   \author{S.~M. Tzaninis}
\address{Department of Statistics and Insurance Science\\ University of
	Piraeus\\ 80 Karaoli and Dimitriou str.\\ 185 34 Piraeus\\ Greece}
\email{stzaninis@unipi.gr}
\thanks{}

\date{\today}

\begin{document}

%\margn{t1a}

%\tableofcontents
%\newpage

%\margn{{\small{mt2a(dis)\\1NS}}}

\begin{abstract}
In the present paper the change of measures technique for compound mixed renewal processes, developed in Tzaninis \& Macheras \cite{mt2}, is applied to the ruin problem in order to compute  the ruin probability and to find upper and lower bounds for it.
\smallskip

\smallskip

\noindent
{\bf AMS Subject Classification (2010):} \noindent Primary 60G55, 91B30 ; secondary 28A35, 60A10, 60G44, 60K05.
\smallskip
\end{abstract}

\maketitle

 \section{Introduction}\label{intro}
The change of measures technique  has been successfully applied to various theories such as 
%large deviations (Ridder and Walrand   \cite{rw} 1992, Schwartz and Weiss \cite{sw} 1995), 
queues and fluid flows (Asmussen \cite{as1},1994, \cite{as2} 1995, Palmowski and Rolski \cite{pr1} 1996; \cite{pr2} 1998), ruin theory (Dassios and Embrechts \cite{de} 1989, Asmussen \cite{as1} 1996, Asmussen and Albrecher \cite{as3} 2010, Schmidli \cite{scm1} 1996, \cite{scm2} 1997, \cite{scm} 2017), simulation (Boogaert and De Waegenaere \cite{bw} 1990, Ridder \cite{rid} 1996) and pricing of insurance risks (premium calculation principles) (Delbaen and Haezendonck \cite{dh} 1989,  Lyberopoulos and Macheras \cite{lm3} 2019, Macheras and Tzaninis \cite{mt3} 2020, \cite{mt2} 2020). The process of interest is usually Markovian and, under a suitably chosen new probability measure, it is again  a Markov process with some ``nicer" desired 
properties.    
 
In \cite{mt2}, the same problem was investigated  for the class of compound mixed renewal processes, % $S$, 
which are not, in general, Markov ones. In the same paper, given a compound mixed renewal process $S$ under $P$, a full characterization of all probability measures $Q$, which are progressively equivalent to $P$ and preserve the type of $S$, but with some better desired properties, was provided, see \cite{mt2}, Theorem 4.5 and Corollary 4.8 as well as Proposition 4.15. Note that the martingales $L^r$ and the probability measures $Q^r$ appearing in \cite{scm} are special instances of the martingales $\wt M^{(\gamma)}(\theta)$ and of the probability measures $Q_\theta$ of Theorem 4.5 from \cite{mt2}. 

Part of \cite{mt2}, Proposition 4.15, is Proposition \ref{Schmidli}, formulated for the purposes of the present paper and being the starting point for  applications to the ruin problem. Proposition \ref{Schmidli}, as well as Proposition \ref{slln}, extend the corresponding results for the renewal risk model (see e.g. \cite{scm}, Lemmas 8.4 and 8.6, respectively) to the compound mixed renewal processes.

A consequence of Proposition \ref{Schmidli} is Theorem \ref{clm2}, where it is proven, that  if the net profit condition is fulfilled under an original measure $P$, and $S$ is a  compound mixed renewal process under $P$ then under the new measure resulting from Proposition \ref{Schmidli} the process $S$ will be of the same type, except that the net profit condition will no longer be fulfilled and ruin will  always 
occur within finite time.    

Thus, the ruin problem becomes easier to handle, since by Theorem \ref{clm2} under the new measure the probability of ruin is equal to 1, giving us the opportunity to express the ruin probability under $P$ as a quantity under the new measure, see Proposition \ref{prop2}, and to find upper and lower bounds for it, see Proposition \ref{cor3}.

 \section{Preliminaries}\label{prel}
 
 {\em Throughout this paper, unless stated otherwise, $(\vO,\vS,P)$ is a fixed but arbitrary probability space.} The symbol  $\mathcal L^{1}(P)$ stands for the family of all real-valued $P$-integrable  functions on $\vO$. Functions that are $P$-a.s. equal are not identified. We denote by  $\sigma(\mathcal G)$  the $\sigma$-algebra generated by a family $\mathcal G$ of subsets of $\vO$. Given a topology  $\mf{T}$ on $\vO$ we write ${\mf B}(\vO)$ for its {\bf Borel $\sigma$-algebra} on $\vO$, i.e. the $\sigma$-algebra generated by $\mf{T}$. Our measure theoretic terminology is standard and generally follows  \cite{Co}.   For the definitions of real-valued random variables and random variables we refer to  \cite{Co}, p. 308.   We apply the notation $P_{X}:=P_X(\theta):={\bf{K}}(\theta)$ to mean that $X$ is distributed according to the law ${\bf{K}}(\theta)$, where $\theta\in D\subseteq\R^d$ ($d\in\N$)  is the parameter of the distribution. We denote again by ${\bf K}(\theta)$ the distribution function induced by the probability distribution ${\bf K}(\theta)$.  Notation  ${\bf Ga}(b,a)$, where $a,b\in(0,\infty)$, stands for the law  of gamma  distribution (cf. e.g. \cite{Sc}, p. 180). In particular, ${\bf Ga}(b,1)={\bf Exp}(b)$ stands for the law of exponential distribution. 
 For  two real-valued random variables $X$ and $Y$ we write $X=Y$ $P$-a.s. if $\{X\neq Y\}$ is a $P$-null set. If $A\subseteq\vO$, then $A^c:=\vO\setminus A$, while $\chi_A$ denotes the indicator (or characteristic) function of the set $A$.  For a map $f:D\longrightarrow E$ and for a non-empty set $A\subseteq D$ we denote  by $f\upharpoonright A$ the restriction of $f$ to $A$. 
   We write $\E_P[X\mid\F]$ for a version of a conditional expectation (under $P$) of $X\in\L^{1}(P)$ given a $\sigma$-subalgebra $\mathcal{F}$ of $\vS$.  
 For $X:=\chi_E\in\mathcal{L}^1(P)$ with $E\in\vS$ we set $P(E\mid\mathcal{F}):=\E_P[\chi_E\mid\mathcal{F}]$. For the unexplained terminology of Probability and Risk Theory we refer to \cite{Sc}. \smallskip
 
 Given two measurable spaces $(\vO,\vS)$ and $(\vY,H)$, a  function $k$ from $\vO\times H$ into $[0,1]$ is a {\bf $\vS$-$H$-Markov kernel} if it has the following properties:
 \begin{itemize}
 	\item[{\bf(k1)}] The set-function $B\longmapsto k(\omega,B)$ is a probability measure on $H$ for any fixed $\omega\in\vO$.
 	\item[{\bf(k2)}] The function $\omega\longmapsto k(\omega,B)$ is $\vS$-measurable for any fixed $B\in{H}$.
 \end{itemize}
 In particular, given a real-valued random variable $X$ on $\vO$ and a $d$-dimensional random vector $\vT$ on $\vO$, a {\bf conditional distribution of $X$ over $\vT$} is a $\sigma(\vT)$-$\B$-Markov kernel denoted by $P_{X\mid\vT}:=P_{X\mid\sigma(\vT)}$ and satisfying for each $B\in\B$ condition
 $$
 P_{X\mid\vT}(\bullet,B)=P(X^{-1}[B]\mid\sigma(\vT))(\bullet)
 \quad{P}\uph\sigma(\vT)-\mbox{a.s.}.
 $$
 
 Clearly, for every $\B_d$-$\B$-Markov kernel $k$, the map $K(\vT)$ from $\vO\times\B$ into $[0,1]$ defined by means of
 $$
 K(\vT)(\omega,B):=(k(\bullet,B)\circ\vT)(\omega)
 \quad\mbox{for any}\;\;(\omega,B)\in \vO\times\B
 $$
 is a $\sigma(\vT)$-$\B$-Markov kernel. Then for $\theta=\vT(\omega)$ with $\omega\in\vO$ the probability measures $k(\theta,\bullet)$ are distributions on $\B$ and so we may write $\mathbf{K}(\theta)(\bullet)$ instead of $k(\theta,\bullet)$. Consequently, in this case $K(\vT)$ will be denoted by $\mathbf{K}(\vT)$.

 For any real-valued random variables $X, Y$ on $\vO$ we say that $P_{X|\vT}$ and $P_{Y|\vT}$ are $P\upharpoonright\sigma(\vT)$-equivalent and we write $P_{X|\vT}=P_{Y|\vT}$ $P\upharpoonright\sigma(\vT)$-a.s.,  if there exists a $P$-null set $M\in\sigma(\vT)$ such that for any $\omega\notin M$ and $B\in \B$ the equality $P_{X|\vT}(\omega,B)=P_{Y|\vT}(\omega,B)$ holds true.\smallskip
 
For the definition of a $P$--conditionally (stochastically) independent process over $\sigma(\vT)$ as well as of a $P$--conditionally identically distributed process over $\sigma(\vT)$ we refer to \cite{mt2}.  We say that a process is
 $P$--{\bf conditionally (stochastically) independent or identically distributed given $\vT$}, if it is conditionally independent or identically distributed over the $\sigma$-algebra  $\sigma(\vT)$.

 {\em For the rest of the paper we simply write ``conditionally" in the place of ``conditionally given $\vT$” whenever conditioning refers to  $\vT$.}

 {\em Henceforth,  unless stated otherwise, $(\vY,H):=((0,\infty),\B(\vY))$  and $\vT$ is a $d$--dimensional  random vector on $\vO$ with values on $D\subseteq\R^d$ ($d\in\N$).} 
 \section{A Change of Measures Technique for Compound Mixed Renewal Processes}\label{CRPPEM}

We first recall some additional background material, needed in this section.

A family $N:=\Nt$ of  random variables from $(\vO,\vS)$ into $(\ov{\R}, \B(\ov\R))$ is called a {\bf counting} (or { \bf claim number}) {\bf process}, if there exists a $P$--null set $\vO_N\in\vS$ such that the process $N$ restricted on $\vO\setminus\vO_N$ takes values in $\N_0\cup\{\infty\}$, has right-continuous paths, presents jumps of size (at most) one, vanishes at $t=0$ and increases to infinity. Without loss of generality we may and do assume, that $\vO_N=\emptyset$. Denote by $T:=\Tn$ and $W:=\Wn$  the {\bf{(claim) arrival process}} and  {\bf (claim) interarrival process}, respectively (cf. e.g. \cite{Sc}, Section 1.1, page 6 for the definitions) associated with $N$. Note also that every arrival process induces a counting process, and vice versa (cf. e.g. \cite{Sc}, Theorem 2.1.1). \smallskip

Furthermore, let $X:=\Xn$ be a sequence of positive  real-valued random variables on $\vO$, and for any $t\geq 0$ define
$$
S_t:=\begin{cases}  \sum^{N_t}_{k=1}X_k &\text{if}\;\; t>0; \\ 0 &\text{if}\;\; t=0. \end{cases}
$$
Accordingly, the sequence $X$ is said to be the {\bf claim size process}, and the family $S:=\St$ of real-valued random variables on $\vO$ is said to be the {\bf aggregate  claims process induced by} the pair $(N,X)$.  Recall that a pair $(N,X)$ is called a {\bf risk process}, if $N$ is a counting process, $X$ is $P$--i.i.d. and the processes $N$ and $X$ are $P$--independent (see \cite{Sc}, Chapter 6, Section 6.1).\smallskip

The following definition has been introduced in \cite{lm6z3}, Definition 3.1, see also \cite{mt1}, Definition 3.2(b).
\begin{df}
	\label{mrp}
 	\normalfont
	A counting process $N$  is  a $P$--{\bf mixed renewal process with mixing parameter $\vT$ and  interarrival time conditional distribution $\bf{K}(\vT)$} (written $P$--MRP$(\bf{K}(\vT))$ for short), if the induced  interarrival process $W$ is $P$--conditionally independent and 
	$$
	\forall\;n\in\N\qquad  [P_{W_n\mid\vT}=\bf{K}(\vT)\quad P\upharpoonright\sigma(\vT)\text{-a.s.}].
	$$
	In particular, if the distribution $P_\vT$ of $\vT$ is degenerate at some point $\theta_0\in D$, then the counting process $N$ becomes  a $P$--{\em renewal process with  interarrival time distribution $\bf{K}(\theta_0)$} (written $P$--RP$(\bf{K}(\theta_0))$ for short).
\end{df}

Accordingly,  an aggregate claims process $S$ induced by a $P$--risk process $(N,X)$  such that $N$ is a $P$--MRP$(\bf{K}(\vT))$ is called  a {\bf compound mixed renewal process  with parameters ${\bf K}(\vT)$ and $P_{X_1}$} ($P$--CMRP$({\bf K}(\vT),P_{X_1})$ for short). In particular, if $P_\vT$ is degenerate at $\theta_0\in D$, then $S$ is called a {\bf compound  renewal process with parameters ${\bf K}(\theta_0)$ and $P_{X_1}$} ($P$--CRP$({\bf K}(\theta_0),P_{X_1})$ for short).\smallskip

{\em Throughout what follows we denote again by ${\bf K}(\vT)$ and ${\bf K}(\theta)$ the conditional distribution function and the distribution function  induced by the conditional probability distribution ${\bf K}(\vT)$ and the  probability distribution ${\bf K}(\theta)$, respectively.}\smallskip

 The following conditions will be useful for our investigations:
 \begin{itemize}
 	\item[{\bf(a1)}] the pair $(W,X)$ is $P$--conditionally independent;
 	\item[{\bf(a2)}] the random vector $\vT$ and the process $X$ are $P$--(unconditionally) independent.
 \end{itemize}
 {\em Next, whenever condition (a1) and (a2) holds true we shall write that the quadruplet $(P,W,X,\vT)$ or (if no confusion arises) the probability measure $P$ satisfies (a1) and (a2), respectively}.\smallskip

Since conditioning is involved in the definition of (compound) mixed renewal processes, it is natural to expect that regular conditional probabilities (or disintegrations) will play a key. To this purpose we recall the following definition.
 
 \begin{df}\label{rcp}
 	 \normalfont
 	Let $(\vY,H,R)$ be a probability probability space. A family $\{P_y\}_{y\in\vY}$ of probability measures on $\vS$ is called a {\bf regular conditional probability} (rcp for short) of $P$ over $R$  if
 	\begin{itemize}
 		\item[{\bf(d1)}] 
 		for each $E\in\vS$ the map $y\longmapsto P_y(E)$ is $H$--measurable;
 		\item[{\bf(d2)}]
 		$\int P_{y}(E)\,R(dy)=P(E)$ for each $E\in\vS$.
 	\end{itemize}
  \end{df}

\begin{comment}
We could use the term of {\em disintegration} instead, but it is better to reserve that term to the general case when$P_y$'s may be defined on different domains (see\cite{pa}).
\end{comment}

 	If $f:\vO\longrightarrow\vY$ is an inverse-measure-preserving function (i.e. $P(f^{-1}(B))=R(B)$ for each $B\in{H}$), a rcp $\{P_{y}\}_{y\in\vY}$ of $P$ over $R$ is called {\bf consistent} with $f$ if, for each $B\in{H}$, the equality $P_{y}(f^{-1}(B))=1$ holds for $R$--almost every $y\in B$.

 {\em From now on, unless stated otherwise, the family $\{P_\theta\}_{\theta\in D}$ is a rcp of $P$ over $P_\vT$ consistent with $\vT$.}
 
 Regular conditional probabilities seem to have a bad  reputation  when it comes to applications, and that is probably due to the fact that their own existence is not always guaranteed (see \cite{lmt1}, Examples 4 and 5). Nevertheless, as the spaces used in applied Probability Theory are mainly Polish ones, such rcps  always exist (see \cite{fa}, Theorem 6), and in fact they can be explicitly constructed for the class of (compound) mixed renewal processes (see \cite{mt2}, Proposition 4.1).

 We write $\F:=\{\F_t\}_{t\in\R_+}$, where $\F_t:=\sigma(\F^S_t\cup\sigma(\vT))$, for the canonical filtration generated by $S$ and $\vT$, $\F^S_\infty:=\sigma(\bigcup_{t\in\R_+}\F^S_t)$ and $\F_\infty:=\sigma(\F^S_\infty\cup\sigma(\vT))$.  For the definition of a $(P,\mathcal{Z})$--martingale in $\mathcal{L}^1(P)$, where $\mathcal{Z}=\{\mathcal Z_t\}_{t\in\R_+}$ is a filtration for $(\vO,\vS)$, we refer to \cite{Sc}, p. 25. A $(P,\mathcal{Z})$--martingale $\{Z_t\}_{t\in\R_+}$ in $\mathcal{L}^1(P)$ is {\bf $P$--a.s. positive},  if $Z_t$ is $P$--a.s. positive for each $t\geq 0$. For $\mathcal Z=\F$ we write ``martingale in $\mathcal{L}^1(P)$" instead of ``$(P,\mathcal{Z})$--martingale in $\mathcal{L}^1(P)$", for simplicity.

 \begin{symbs}\label{symb2}
  \normalfont 
 	{\bf(a)}  The class of all real--valued $\B(\vY)$--measurable functions $\ga$  such that $\E_{P}\left[e^{\ga(X_{1})}\right]=1$ will be denoted by $\F_{P}:=\F_{P,X_1,\ln}$. The class of all real--valued  $\B(D)$--measurable functions $\xi$ on $D$ such that $P_{\vT}(\{\xi>0\})=1$  and $\E_P[\xi(\vT)]=1$ is denoted by $\mathcal{R}_+(D):=\mathcal{R}_+(D,\mathfrak{B}(D), P_{\vT})$.\smallskip
 	
 	{\bf (b)} Denote by $\mathfrak{M}^k(D)$ ($k\in\N$)  the class of all $\B(D)$--$\B(\R^k)$--measurable functions on $D$.  For each $\rho\in\mf{M}^k(D)$ the class of all probability measures $Q$ on $\vS$ satisfying  (a1)  and  (a2), are {\bf progressively equivalent} to $P$, i.e. $Q\uph\F_t\sim P\uph\F_t$ for any $t\geq 0$ (in the sense of absolute continuity),   and such that $S$ is a $Q$--CMRP$({\bf \Lambda}(\rho(\vT)),Q_{X_1})$ is denoted by $\M_{S,{\bf\Lambda}(\rho(\vT))}:=\M_{S,{\bf \Lambda}(\rho(\vT)),P,{X_1}}$. In the special case $d=k$ and $\rho:=id_D$ we write $\M_{S,{\bf \Lambda}(\vT)}:=\M_{S,{\bf \Lambda}(\Ttheta)}$  for simplicity. \smallskip
 	
 	{\bf (c)} For  given $\rho\in\mathfrak M^k(D)$ and  $\theta\in D$,  denote by ${\M}_{S,{\bf\Lambda}(\ttheta)}$  the class of all probability measures $Q_\theta$ on $\vS$, such that $Q_\theta\uph\F_t\sim P_\theta\uph\F_t$ for any $t\in\R_+$ and $S$ is a $Q_\theta$-CRP$({\bf \Lambda}(\ttheta),(Q_\theta)_{X_1})$.  
 \end{symbs}
 
 {\em From now on, unless stated otherwise,  $P\in\M_{S,{\bf K}(\vT)}$ is the initial probability measure under which $S$ is a $P$--CMRP$({\bf K}(\vT),P_{X_1})$.} \smallskip

It follows a result to show that the martingales $L^r:=\{L_t^r\}_{t\in\R_+}$ and the measures $Q^r$ appearing in \cite{scm}, Chapter 8, Section 8.3, are special instances of the martingales $\wt{M}^{(\ga)}(\theta)$ and the measures $Q_{\theta}$, respectively, of the main result of \cite{mt2}, i.e. Theorem 4.5. 

\begin{rem}\label{lem35} 
	\normalfont
For any $r\in\R_+$ such that $\E_P[e^{rX_1}]<\infty$ and for any $\theta\in{L}_P^c$,  where $L_P$ is the $P_\vT$--null appearing in \cite{mt2}, Proposition 3.3,  let $\kappa_{\theta}(r)$ be the unique solution to the equation 
 \begin{equation}\label{lem35a}
M_{X_1}(r)\cdot (M_{\theta})_{W_1}\bigl(-\kappa_{\theta}(r)-c(\theta)\cdot{r}\bigr)=1,
\end{equation}
where $M_{X_1}$ and $(M_{\theta})_{W_1}$ are the moment generating function of $X_1$ and $W_1$ under the measures $P$ and $P_{\theta}$, respectively. (Such a solution exists by e.g. \cite{rss}, Lemma 11.5.1(a)). Define the function $\kappa:D\times\R_+\longrightarrow\R$ by means of 
\[
\kappa(\theta,r):=\kappa_{\theta}(r)\quad\mbox{for any}\quad (\theta,r)\in{D}\times\R_+,
\]
and for fixed $r\in\R_+$ denote by $\kappa_{\vT}$ the random variable defined by the formula
\[
\kappa_{\vT}(r)(\omega):=\kappa_{\vT(\omega)}(r)\quad\mbox{for any}\quad\omega\in\vO.
\]
Then, due to \cite{mt2}, Lemma 4.13,  $\kappa_{\vT}(r)$ is the $P\uph\sigma(\vT)$-a.s. unique solution to the equation
 \begin{equation}\label{lem35b}
M_{X_1}(r)\cdot \E_P\bigl[e^{-\bigl(\kappa_{\vT}(r)+c(\vT)\cdot{r}\bigr)W_1}\mid\vT\bigr]=1\quad{P}\uph\sigma(\vT)-\mbox{a.s.}.
\end{equation}
\end{rem}

The following proposition is a part of Proposition 4.15 from \cite{mt2}. Since it is the basic tool for the proofs of our results, we restate it exactly in the form needed for our purposes.

\begin{prop}\label{Schmidli}
For any $r\in\R_+$ such that $\E_P[e^{rX_1}]  <\infty$,  and for any $\theta\notin{L}_P$, let $\kappa_{\theta}(r)$ be the unique solution to the equation \eqref{lem35a}, and let $\kappa_{\vT}(r)$ be as in Remark \ref{lem35}. Fix on arbitrary $r\in\R_+$ as above and let $\rho\in\mathfrak{M}^k(D)$ be given.
	
For each pair $(\ga,\xi)\in \F_P\times\mathcal{R}_+(D)$ with $\gamma(x):=r\cdot x-\ln\E_{P}[e^{r \cdot X_1}]$ for any $x\in\vY$,  there exists a unique probability measure $Q:=Q^r\in\M_{S,{\bf\Lambda}(\rho(\vT))}$, where 
\[
{\mathbf\Lambda (\Ttheta)}(B):=\frac{\E_{P}[\chi_{W_1^{-1}[B_2]}\cdot e^{-(\kappa_\vT(r)+c(\vT)\cdot r)\cdot W_1}\mid\vT]}{\E_{P}[e^{-(\kappa_\vT(r)+c(\vT)\cdot r)\cdot W_1}\mid\vT]}\quad{P}\uph\sigma(\vT)-\mbox{a.s.}
\]
for any $B\in\B(\vY)$, determined by  condition 
\[
Q(A)=\int_AM_t^{(\ga)}(\vT)dP\quad\mbox{for all}\quad 0\leq{u}\leq{t}\quad\mbox{and}\quad A\in\mathcal{F}_u\leqno(RRM_{\xi})
\]
with the martingale $M^{(\ga)}(\vT)$ in $\mathcal{L}^1(P)$ fulfilling condition 
\[
M_t^{(\ga)}(\vT)=\xi(\vT)\cdot \wt{M}_t^{(\ga)}(\vT)\quad P\uph\sigma(\vT)-\mbox{a.s.}.
\]

Moreover, there exist an essentially unique rcp $\Qtd:=\{Q_{\theta}^r\}_{\theta\in{D}}$ of $Q$ over $Q_\vT$  consistent with $\vT$ and a $P_\vT$-null set $L_{\ast\ast}\in\B(D)$, satisfying for any $\theta\notin L_{\ast\ast}$ conditions $Q_\theta\in{\M}_{S,{\bf \Lambda}(\ttheta)}$ and 
\begin{equation}
\tag{$RRM_\theta$}
Q_\theta(A)=\int_{A} \wt M^{(\ga)}_{t}(\theta)\,dP_\theta\quad\text{for all}\,\,\,0\leq u\leq t\,\,\text{and}\,\, A\in\F_{u}
\label{rcp2}
\end{equation} 
 with the martingale $\wt{M}_t^{(\ga)}(\theta)$ in $\mathcal{L}^1(P_\theta)$ fulfilling condition 
\[
\wt M_t^{(\ga)}(\theta) =e^{r\cdot S_t-(\kappa_\theta(r)+c(\theta)\cdot r)\cdot T_{N_t} +\ln \E_P[e^{r\cdot X_1}]}\cdot \frac{\int^{\infty}_{t-T_{N_t}} e^{-(\kappa_\theta(r)+c(\theta)\cdot r)\cdot w} \, (P_\theta)_{W_1}(dw)}{1-{\bf{K}}(\theta)(t-T_{N_t})}.
\]

In particular, if $P_{W_1}$ is absolutely continuous with respect to the Lebesgue measure $\lambda$ restricted to $\mf{B}([0,1])$, then the martingale $L^r(\theta):=\{L^r_t(\theta)\}_{t\in\R_+}$ for $r\in\R_+$, appearing in \cite{scm}, Lemma 8.4, coincides with the martingale $\wt{M}^{(\ga)}(\theta)$ in $\mathcal{L}^1(P_\theta)$ for any $\theta\notin{L}_{\ast\ast}$, and for any $t\in\R_+$ condition 
\[
M_t^{(\ga)}(\vT)=\xi(\vT)\cdot L_t^r(\vT)
\]
holds $P\uph\sigma(\vT)$--a.s. true. 
\end{prop}

\begin{lem}\label{lem4}
  For any $r\in\R_+$, $\theta\notin L_{\ast\ast}$, $Q^r_\theta$ and $\kappa_{\theta}(r)$ as in Proposition \ref{Schmidli} condition 
\[
\kappa^{\prime}_\theta(r)=\frac{\E_{Q^r_\theta}[X_1]}{\E_{Q^r_\theta}[W_1]} -c(\theta),  
\]
holds true.
\end{lem} 	 
 	 
\begin{proof} 
 Fix on arbitrary $r\in\R_+$ and $\theta\notin L_{\ast\ast}$ as in Proposition \ref{Schmidli}.

 Since $L_P\subseteq L_{\ast\ast}$ by \cite{mt2}, Theorem 4.5, it follows by \cite{mt2}, Proposition 3.3,  that we can  rewrite condition \eqref{lem35a} in the form 
 \begin{equation}
 (M_\theta)_{X_1}(r)\cdot (M_\theta)_{W_1}(-c(\theta)\cdot r-\kappa_\theta(r))=1.
 \label{acf2a}
 \end{equation} 
  Differentiation with respect to $r$ gives
 \begin{equation}  
\begin{split}
& \bigl((M_\theta)_{X_1}(r))^{\prime}\cdot (M_\theta)_{W_1}(-c(\theta)\cdot r-\kappa_\theta(r)\bigr)\\
&\qquad+(M_\theta)_{X_1}(r)\cdot \left((M_\theta)_{W_1}(-c(\theta)\cdot r-\kappa_\theta(r))\right)^{\prime}\cdot(-c(\theta)-\kappa^{\prime}_\theta(r))=0
 \end{split}
\label{mgf1}
\end{equation}
for all $r$ in a neighbourhood of $0$. The expectations $\E_{Q_{\theta}}^r[X_1]$ and $\E_{Q_{\theta}}^r[W_1]$ are given by
%\begin{equation}
\[
\E_{Q^r_\theta}[X_1]=\E_{P_\theta}\left[X_1\cdot\frac{e^{r\cdot X_1}}{\E_{P_\theta}[e^{r\cdot X_1}]}\right]=\frac{\E_{P_\theta}\left[X_1\cdot e^{r\cdot X_1}\right]}{\E_{P_\theta}[e^{r\cdot X_1}]}=\frac{	(M^\theta_{X_1}(r))'}{	M^\theta_{X_1}(r)}
\]
and
\begin{align*}
\E_{Q^r_\theta}[W_1]&=\E_{P_\theta}\left[W_1\cdot\frac{e^{-(r\cdot c(\theta)+\kappa_\theta(r))\cdot W_1}}{\E_{P_\theta}[e^{-(r\cdot(\theta) c+\kappa_\theta(r))\cdot W_1}]}\right]=\frac{\E_{P_\theta}\left[W_1\cdot e^{-(r\cdot c(\theta)+\kappa_\theta(r))\cdot W_1}\right]}{\E_{P_\theta}[e^{-(r\cdot c(\theta)+\kappa_\theta(r))\cdot W_1}]}\nonumber\\
&=\frac{\bigl((M_\theta)_{W_1}(-r\cdot c(\theta)-\kappa_\theta(r)\bigr))^{\prime}}{\bigl((M_\theta)_{W_1}(-r\cdot c(\theta)-\kappa_\theta(r)\bigr)},
%\label{eq2a}
\end{align*}
respectively, implying along with condition \eqref{mgf1} that
 \begin{align*}
&\E_{Q^r_\theta}[X_1]\cdot (M_\theta)_{X_1}(r)\cdot(M_\theta)_{W_1}(-c(\theta)\cdot r-\kappa_\theta(r))\\
&\quad +(M_\theta)_{X_1}(r)\cdot \E_{Q^r_\theta}[W_1]\cdot(M_\theta)_{W_1}(-c(\theta)\cdot r-\kappa_\theta(r))\cdot(-c(\theta)-\kappa^{\prime}_\theta(r))=0.
\end{align*}
The latter together with condition \eqref{acf2a} gives
\[
\E_{Q^r_\theta}[X_1] + \E_{Q^r_\theta}[W_1] \cdot(-c(\theta)-\kappa^{\prime}_\theta(r))=0,
\]
completing the proof.
\end{proof}

 Let $S$ be the aggregate claims process induced by the counting process $N$ and the claim size process $X$. Fix on arbitrary $u\in\vY$ and $t\in\R_+$, and define the function $r_t^u:\vO\times{D}\longrightarrow\R$ by means of  
 $r_t^u(\omega,\theta):=u+c(\theta)\cdot{t}-S_t(\omega)$ for any $(\omega,\theta)\in\vO\times{D}$, where $c$ is a positive $\B(D)$--measurable function. For arbitrary but fixed $\theta\in{D}$, the process $r^u(\theta):=\{r_t^u(\theta)\}_{t\in\R_+}$ defined by  $r_t^u(\theta):=r_t^u(\omega,\theta)$ for any $\omega\in\vO$, is called the {\bf reserve process} induced by the {\bf initial reserve} $u$, the {\bf premium intensity} or {\bf premium rate} $c(\theta)$ and the aggregate claims process $S$ (see \cite{Sc}, Section 7.1, pages 155-156 for the definition). The function $\psi_{\theta}$ defined by $\psi_{\theta}(u):=P_{\theta}(\{\inf{r}_t^u(\theta)<0\})$  is called the {\bf probability of ruin} for the reserve process $r^u(\theta)$ with respect to $P_\theta$ (see \cite{Sc}, Section 7.1, page 158 for the definition).

Define the real-valued functions $r_t^u(\vT)$ and $R_t^u$ on $\vO$ by means of  $
r_t^u(\vT)(\omega):=r_t^u(\omega,\vT(\omega))$ for any $\omega\in\vO$,  and $R_t^u:=r_t^u\circ(id_{\vO}\times\vT)$, respectively. The process $R^u:=\{R_t^u\}_{t\in\R_+}$ is called the {\bf reserve process} induced by the  initial reserve $u$, the {\bf stochastic premium intensity} or {\bf stochastic premium rate} $c(\vT)$ and the aggregate claims process $S$. The function $\psi$ defined by 
$\psi(u):=P(\{\inf{R}_t^u<0\})$ is called the {\bf probability of ruin} for the reserve process $R^u$ with respect to $P$.

The following proposition extends Lemma 8.6 of \cite{scm}.  

\begin{prop}\label{slln}
  For any $r\in\R_+$, $\theta\notin{L}_{\ast\ast}$, $Q^r_{\theta}$ and $\kappa_{\theta}(r)$ as in Proposition \ref{Schmidli}, the following statements hold true:
\begin{enumerate}
\item
$\lim_{t\rightarrow\infty}\frac{r^u_t(\theta)-u}{t}=-\kappa^{\prime}_\theta(r)\quad Q_{\theta}^r-\mbox{a.s.}$;
\item
$\lim_{t\rightarrow\infty}\frac{R^u_t-u}{t}=-\kappa^{\prime}_\vT(r)\quad Q^r-\mbox{a.s.}$.
\item
if there exists a $P_\vT$--null set $\wh{L}_1$ in $\mf{B}(D)$ such that $P_\theta=Q^r_\theta$ for any $\theta\notin \wh{L}_1$, then the measures $P$ and $Q^r$ are equivalent on $\F_\infty$;  
\item 
if there exists a $P_\vT$--null set $\wh{L}_2$ in $\mf{B}(D)$ such that  $P_\theta\neq Q^r_\theta$  for any $\theta\notin \wh{L}_2$, then the measures $P$ and $Q^r$ are singular on $\F_\infty$, i.e. there exists a set $E\in\F_{\infty}$ such that $P(E)=0$ if and only if $Q^r(E)=1$.
\end{enumerate}
\end{prop}

\begin{proof}
Fix on arbitrary $r\in\R_+$ as in Proposition \ref{Schmidli}.

Ad (i): Let us fix on arbitrary $\theta\not\in{L}_{\ast\ast}$, and note that $L_P\subseteq{L}_{\ast\ast}$ by \cite{mt2}, Theorem 4.5.  Since $S$ is a $Q^r_\theta$--CRP by \cite{mt2}, Proposition 3.3, we get by the strong law of large numbers  that 
 	 		\[
 	 		\lim_{t\rightarrow\infty}\frac{S_t}{t}=\frac{\E_{Q^r_\theta}[X_1]}{\E_{Q^r_\theta}[W_1]}\quad Q^r_\theta\text{--a.s.}
 	 		\] 
 	 		(cf. e.g. \cite{gut}, Section 1.2, Theorem 2.3), or equivalently that 
 	 		\[
 	 		\lim_{t\rightarrow\infty}\frac{r^u_t(\theta)-u}{t}=\lim_{t\rightarrow\infty}\frac{c(\theta)\cdot t-S_t}{t}=c(\theta)-\lim_{t\rightarrow\infty}\frac{S_t}{t}=c(\theta)-\frac{\E_{Q^r_\theta}[X_1]}{\E_{Q^r_\theta}[W_1]}\qquad Q^r_\theta\text{--a.s.},
 	 		\]
 	 		implying along with Lemma \ref{lem4}, assertion (i). 

Ad (ii):  			
Consider the function  $v:=\chi_{\left\{\lim_{t\to\infty}\frac{r^u_t-u}{t}=-\kappa_{\bullet}^{\prime}(r)\right\}}:\vO\times D\longrightarrow [0,1]$ and put $g:=v\circ (id_\vO\times\vT)=\chi_{\left\{\lim_{t\to\infty}\frac{r^u_t(\vT)-u}{t}=-\kappa_\vT^\prime(r)\right\}}$. 
Since $v\in\L^1(M)$, where $M:=P\circ (id_\vO\times\vT)^{-1}$, we may apply \cite{lm1v}, Proposition 3.8(i), to get that 
 		\[
 		\E_{Q^r}\left[g\mid\vT\right]=\E_{Q^r_\bullet}\left[v^{\bullet}\right]\circ \vT\quad Q^r\uph\sigma(\vT)\text{--a.s.}
 		\] 
 		or equivalently
 		\[
 		Q^r\left(\left\{\lim_{t\to\infty}\frac{r^u_t(\vT)-u}{t}=-\kappa_{\vT}^{\prime}(r)\right\}\mid\vT\right)=Q^r_\bullet\left(\left\{\lim_{t\to\infty}\frac{r^u_t(\bullet)-u}{t}=-\kappa_\bullet^{\prime}(r)\right\}\right)\circ \vT\quad Q^r\uph\sigma(\vT)\text{--a.s.}.
 		\] 
 		Then for any  $F\in\B(D)$ we get   
 		\begin{align*}
 		&\int_{\vT^{-1}[F]}Q^r\left(\left\{\lim_{t\to\infty}\frac{r^u_t(\vT)-u}{t}=-\kappa_\vT^\prime(r)\right\}\mid\vT\right)\, dQ^r\\
 		&=\int_{\vT^{-1}(F)}Q^r_\bullet\left(\left\{\lim_{t\to\infty}\frac{r^u_t(\bullet)-u}{t}=-\kappa_\bullet^\prime(r)\right\}\right)\circ \vT\, dQ^r\\
 		&=\int_{F} 	Q^r_\theta\left(\left\{\lim_{t\to\infty}\frac{r^u_t(\theta)-u}{t}=-\kappa_\theta^\prime(r)\right\}\right)\, Q^r_\vT(d\theta)\\
 		&=\int_{F\cap L^c_{\ast\ast}}Q^r_\theta\left(\left\{\lim_{t\to\infty}\frac{r^u_t(\theta)-u}{t}=-\kappa_\theta^\prime(r)\right\}\right)\, Q^r_\vT(d\theta)\\
 		&=\int_{\vT^{-1}(F)}\,dQ^r,
 		\end{align*}
 		where the last equality follows by (i); hence 
 		\[
 		Q^r\left(\left\{\lim_{t\to\infty}\frac{r^u_t(\vT)-u}{t}=-\kappa^\prime_\vT(r)\right\}\mid\vT\right)=1\quad Q^r\uph\sigma(\vT)\text{--a.s.},
 		\]
 		implying
 		\[
 		Q^r\left(\left\{\lim_{t\to\infty}\frac{r^u_t(\vT)-u}{t}=-\kappa^\prime_\vT(r)\right\}\right)=\int Q^r\left(\left\{\lim_{t\to\infty}\frac{r^u_t(\vT)-u}{t}=-\kappa^\prime_\vT(r)\right\}\mid\vT\right)\,dQ^r= 1,
 		\]
that is assertion (ii) holds true.
 
 The proof of the statements (iii) and (iv) follow by Proposition \ref{Schmidli} together with \cite{mt2}, Proposition 3.11. 		
\end{proof}

 \section{Applications to the Ruin Problem}\label{char}

In this section we present the main result of the paper, where  an explicit formula for the probability of ruin for the reserve process in the case of compound mixed renewal processes is proven. Before we formulate it, we need to establish the validity of the following theorem, which is  a consequence of Proposition \ref{Schmidli}, and allows us to construct a probability measure $Q^{R^\ast}$, being singular to the original probability measure $P$ and such that ruin occurs $Q^{R^\ast}$--a.s..

\begin{thm}\label{clm2} 
Let $r\in\R_+$ and $\theta\notin{L}_{\ast\ast}$ be as in Proposition \ref{Schmidli}. If for any $\theta\notin L_{\ast\ast}$ the  net profit condition  is fulfilled with respect to $P_\theta$, i.e.
\[
c(\theta)>\frac{\E_{P_\theta}[X_1]}{\E_{P_\theta}[W_1]},
\] 
then there exists an adjustment coefficient   $R(\theta)\in\vY$  with respect to $P_\theta$. In particular, if the $\sup_{\theta\in L^c_{\ast\ast}} R(\theta)=:R^\ast$ exists in $\vY$, $\E_P\bigl[e^{R^{\ast}\cdot{X}_1}\bigr]<\infty$ and $P_{W_1}$ is absolutely continuous with respect to the Lebesgue measure $\lambda$ restricted to $\mf{B}([0,1])$, then for any  pair $(\ga,\xi)$ as in Proposition \ref{Schmidli}  there exist a unique probability measure $Q^{R^\ast}\in\mathcal{M}_{S,\mathbf{\Lambda}(\rho(\vT))}$ determined by condition ($RRM_{\xi}$), and a rcp $\{Q_\theta^{R^\ast}\}_{\theta\in{D}}$ of $Q^{R^\ast}$ over $Q^{R^\ast}_\vT$ consistent with $\vT$ satisfying conditions $Q^{R^\ast}_\theta\in\mathcal{M}_{S,\mathbf{\Lambda}(\rho(\theta))}$ and ($RRM_\theta$) for any $\theta\notin{L}_{\ast\ast}$, and such that for any $u>0$ the probabilities of ruin $\psi_\theta^{R^\ast}(u)$ and $\psi^{R^\ast}(u)$ with respect to $Q_\theta^{R^\ast}$ and $Q^{R^\ast}$, respectively,  are equal to 1.
\end{thm}

\begin{proof}
Fix on  arbitrary   $\theta\notin L_{\ast\ast}$ and assume that $c(\theta)>\frac{\E_{P_\theta}[X_1]}{\E_{P_\theta}[W_1]}$. It then follows by e.g. \cite{scm}, page 133, that there exists an adjustment coefficient $R(\theta)$ with respect to $P_{\theta}$.

In particular, assume that $R^\ast\in\vY$, $\E_P\bigl[e^{R^{\ast}\cdot{X}_1}\bigr]<\infty$ and that $P_{W_1}$ is absolutely continuous with respect to the Lebesgue measure $\lambda$ restricted to $\mf{B}([0,1])$. By Proposition \ref{Schmidli}  there exist a unique probability measure $Q^{R^\ast}\in\mathcal{M}_{S,\mathbf{\Lambda}(\rho(\vT))}$ determined by condition ($RRM_{\xi}$), and a rcp $\{Q_\theta^{R^\ast}\}_{\theta\in{D}}$ of $Q^{R^\ast}$ over $Q^{R^\ast}_\vT$ consistent with $\vT$ satisfying conditions $Q^{R^\ast}_\theta\in\mathcal{M}_{S,\mathbf{\Lambda}(\rho(\theta))}$ and ($RRM_\theta$). Because $\kappa^{\prime\prime}_\theta(r)>0$ by e.g. \cite{scm}, p. 133, we get that the function $\kappa_\theta$ is strictly convex, or equivalently  that  $\kappa^{\prime}_\theta$ is strictly increasing. Thus, since by e.g. \cite{scm}, page 133, we have that $\kappa_\theta(0)=\kappa_\theta(R(\theta))=0$ and $\kappa^{\prime}_\theta(0)<0$, it follows that there exists a point $r_0\in(0,R(\theta))$ such that $\kappa^{\prime}_\theta(r_0)=0$; hence $\kappa^{\prime}_\theta(r)>0$ for any $r>r_0$. Because $r_0<R(\theta)\leq R^\ast$ we deduce that $\kappa^{\prime}_\theta(R^\ast)>0$. The latter, along with Lemma \ref{lem4}, yields that   
  \[
 0<\frac{\E_{Q^{R^\ast}_\theta}[X_1]}{\E_{Q^{R^\ast}_\theta}[W_1]} -c(\theta)  \Longleftrightarrow c(\theta)< \frac{\E_{Q^{R^\ast}_\theta}[X_1]}{\E_{Q^{R^\ast}_\theta}[W_1]},
 \]
 implying that the net profit condition is violated with respect to $Q^{R^\ast}_\theta$; hence by \cite{Sc}, Corollary 7.1.4, we obtain   
  \[
 \psi_\theta^{R^\ast}(u)=Q^{R^\ast}_\theta(\{\inf_{t\in\R_+}r^u_t(\theta)< 0\})=1\quad\text{ for any }\,\,u>0,
 \]
implying along with \cite{mt2}, Remark 3.4(b) that  
 \[
 \psi^{R^\ast}(u)=Q^{R^\ast}(\{\inf_{t\in\R_+}R^u_t< 0\})=\int_D \psi_\theta^{R^\ast}(u)\, Q^{R^\ast}_\vT(d\theta)=1 \quad\text{ for any }\,\,u>0,
 \]
completing the whole proof.	
\end{proof}

 It follows an example where the assumptions $R^\ast<\infty$ and $\E_P\bigl[e^{R^\ast{X}_1}\bigr]<\infty$ of Theorem \ref{clm2} hold. 

\begin{ex}\normalfont
Assume that $S$ is a $P$--CMPP$(\vT)$ such that $P_{X_1}={\bf Exp}(\eta)$, $\eta\in\vY$, and $P_\vT={\bf Beta}(a,b)$, $a,b\in(0,\infty)$ i.e.
  \[
  {\bf Beta}(a,b)(B):=\int_B\frac{1}{B(a,b)}\cdot \theta^{a-1}\cdot(1-\theta)^{b-1}\,\lambda(d\theta)\quad\text{ for any }\,\, B\in\B((0,1)).
  \] 
According to \cite{mt2}, Proposition 3.3, there exists a $P_\vT$--null set $L_P\in\B((0,1))$ such that  $S$ is a $P_\theta$--CPP$(\theta)$ for any $\theta\notin L_P$. Fix on an arbitrary $\theta\notin L_P$ and assume that  $c(\theta)=\frac{2\cdot\theta}{\eta\cdot(1+\theta)}$. 
Applying \cite{Sc}, Theorem 7.4.5, we get that $R(\theta)=\eta-\frac{\theta}{c(\theta)}=\frac{\eta\cdot(1-\theta)}{2}\in(0,\eta)$ is an adjustment coefficient with respect to $P_\theta$. Since $\sup_{\theta\in L^c_{P}}R(\theta)=\sup_{\theta\in L^c_{P}} \frac{\eta\cdot(1-\theta)}{2}=\frac{\eta}{2}$, we obtain $R^\ast\in\vY$. Furthermore, $\E_P\bigl[e^{R^{\ast}X_1}]=\frac{\eta}{\eta-R^\ast}<\infty$.
\end{ex}  
   
     Denote by $\tau$ the {\bf ruin time} of the portfolio of an insurance company (cf. e.g. \cite{scm}, p. 84 for the definition).  The next proposition is a consequence of Theorem \ref{clm2}.

\begin{prop}\label{prop2} 
Let $r\in\R_+$ be as in Proposition \ref{Schmidli}, $u\in\vY$ and $\theta\notin{L}_{\ast\ast}$. Denote by $\psi_\theta(u)$ and $\psi(u)$ the probabilities of ruin with respect to the probability measures $P_{\theta}$ and $P$, respectively. Under the assumptions of Theorem  \ref{clm2} the following hold:
\begin{enumerate}
\item
$\psi_\theta(u)=\E_{Q_{\theta}^{R^\ast}}\bigl[e^{R^{\ast}r_\tau^u(\theta)+\kappa_\theta(R^\ast)\cdot\tau}\bigr]\cdot{e}^{-R^{\ast}u}$;
\item
$\psi(u)= \E_{Q^{R^\ast}}\left[\frac{e^{R^\ast\cdot R^u_\tau +\kappa_\vT(R^\ast)\cdot \tau}}{\xi(\vT)} \right]\cdot e^{-R^\ast\cdot u}$.
 \end{enumerate}
\end{prop}

\begin{proof}  
Fix on arbitrary $u\in\vY$

Ad (i): Let $\theta\notin L_{\ast\ast}$ be arbitrary but fixed. Since  $L^{R^\ast}(\theta)$ is an a.s. positive martingale in $\mathcal L^1(P_\theta)$ by Proposition \ref{Schmidli}, and $\tau$ is a stopping time for $\F$, we may apply \cite{scm}, Lemma 8.1, to get
  \begin{align*}
  	& \psi_\theta(u)\\
  	&=\int_{\{\tau<\infty\}}\frac{1}{L_{\tau}^{R^\ast}(\theta)}\,dQ^{R^\ast}_\theta\\
  	&= \int_{\{\tau<\infty\}}  \frac{e^{R^\ast\cdot(r^u_\tau(\theta)-u) -(\kappa_\theta(R^\ast)+c(\theta)\cdot R^\ast)\cdot (\tau-T_{N_\tau}) +\kappa_\theta(R^\ast)\cdot \tau -\ln\E_P[e^{R^\ast\cdot X_1}]}\cdot\left(1-{\bf{K}}(\theta)(\tau-T_{N_\tau})\right)}{\int^{\infty}_{\tau-T_{N_\tau}} e^{-(\kappa_\theta(R^\ast)+c(\theta)\cdot R^\ast)\cdot w} \, (P_\theta)_{W_1}(dw)}\,dQ^{R^\ast}_\theta\\
  	&= \int_{\{\tau<\infty\}}  e^{R^\ast\cdot(r^u_\tau(\theta)-u)   +\kappa_\theta(R^\ast)\cdot \tau}   \,dQ^{R^\ast}_\theta;
  	\end{align*}
  	where the last equality follows from condition \eqref{acf2a} for $r=R^\ast$ and the fact that $\tau-T_{N_\tau}=0$; 
  	hence
  	\[
  	\psi_\theta(u)= \E_{Q^{R^\ast}_\theta}\left[\chi_{\{\tau<\infty\}}\cdot e^{R^\ast\cdot r^u_\tau(\theta) +\kappa_\theta(R^\ast)\cdot \tau}\right]\cdot e^{-R^\ast\cdot u}.
  	\]
Because the probability of ruin with respect to $Q^{R^\ast}_\theta$ is equal to $1$, the previous condition yields
\[
\psi_\theta(u)= \E_{Q^{R^\ast}_\theta}\left[e^{R^\ast\cdot r^u_\tau(\theta) +\kappa_\theta(R^\ast)\cdot \tau}\right]\cdot e^{-R^\ast\cdot u},
\]
that is assertion (i) holds true.
  
Ad (ii): Assertion (i) together with \cite{mt2}, Remark 3.4(b), implies  
\begin{align*}
  	\psi(u)&=\int\psi_\theta(u)\,P_\vT(d\theta)\\
  	&=\int\E_{Q^{R^\ast}_\theta}\left[\frac{e^{R^\ast\cdot r^u_\tau(\theta) +\kappa_\theta(R^\ast)\cdot \tau}}{\xi(\theta)}\right] \,Q^{R^\ast}_\vT(d\theta) \cdot e^{-R^\ast\cdot u}\\
  	&=\int\E_{Q^{R^\ast}}\left[\frac{e^{R^\ast\cdot R^u_\tau +\kappa_\vT(R^\ast)\cdot \tau}}{\xi(\vT)}\mid\vT\right] \,dQ^{R^\ast}  \cdot e^{-R^\ast\cdot u},
  	\end{align*}
  	where the last equality follows from \cite{lm1v}, Proposition 3.8; hence 
\[
\psi(u)= \E_{Q^{R^\ast}}\left[\frac{e^{R^\ast\cdot R^u_\tau +\kappa_\vT(R^\ast)\cdot \tau}}{\xi(\vT)} \right]\cdot e^{-R^\ast\cdot u},
\]
that is assertion (ii) holds true.
\end{proof}

The following result shows that Proposition \ref{prop2} along with Proposition \ref{Schmidli} give us the opportunity to find upper and lower bounds of the probability of ruin under $P$.  

\begin{prop}\label{cor3}
 In the situation of Proposition \ref{prop2} the following holds true:
\begin{enumerate}
\item
$\psi_\theta(u)\geq \E_{Q^{R^\ast}_\theta}\left[e^{R^\ast\cdot r^u_\tau(\theta) }\right]\cdot e^{-R^\ast\cdot u}$;
\item
$\psi(u)\geq \E_{Q^{R^\ast}}\left[\frac{e^{R^\ast\cdot R^u_\tau}}{\xi(\vT)} \right]\cdot e^{-R^\ast\cdot u}$.
\end{enumerate}
In particular, if condition $\E_P\bigl[e^{R^\ast\vT}\bigr]<\infty$ holds and if the function $\xi:D\longrightarrow\R$ is defined by means of
\[
\xi(\theta):=\frac{e^{R^\ast\theta}}{\E_P\bigl[e^{R^\ast\vT}\bigr]}\quad\mbox{for any}\quad \theta\in{D},
\]
then there exist a unique probability measure $\nu^{R^\ast}\in\mathcal{M}_{S,\mathbf{\Lambda}(\rho(\vT))}$ determined by condition ($RRM_{\xi}$) with $R^\ast$ in the place of $r$, and a rcp $\{\nu_\theta^{R^\ast}\}_{\theta\in{D}}$ of $\nu^{R^\ast}$ over $\nu^{R^\ast}_\vT$ consistent with $\vT$ satisfying conditions $\nu^{R^\ast}_\theta\in\mathcal{M}_{S,\mathbf{\Lambda}(\rho(\theta))}$ and ($RRM_\theta$) for any $\theta\notin{L}_{\ast\ast}$, and such that
\[
\psi(u)\leq\E_P\bigl[e^{R^\ast\vT}\bigr]\cdot\E_{\nu^{R^\ast}}\left[ e^{R^\ast\cdot R^u_\tau +\kappa_\vT(R^\ast)\cdot \tau}\right]\cdot{e}^{-R^\ast\cdot{u}}.
\]
\end{prop}

\begin{proof}
Because $\kappa_\theta(R^\ast>0$, statements (i) and (ii) follow by statements (i) and (ii) of Proposition \ref{prop2}, respectively. 

In particular, if for any $r\in\R_+$ condition $\E_P[e^{R^\ast\vT}]<\infty$ holds and $\xi$ is defined as above, the $\xi\in\mathcal{R}_+(D)$ and, due to Proposition \ref{Schmidli}, there exist a unique probability measure $\nu^{R^\ast}\in\mathcal{M}_{S,\mathbf{\Lambda}(\rho(\vT))}$ determined by condition ($RRM_{\xi}$) and a rcp $\{\nu_\theta^{R^\ast}\}_{\theta\in{D}}$ of $\nu^{R^\ast}$ over $\nu^{R^\ast}_\vT$ consistent with $\vT$ satisfying conditions $\nu^{R^\ast}_\theta\in\mathcal{M}_{S,\mathbf{\Lambda}(\rho(\theta))}$ and ($RRM_\theta$) for any $\theta\notin{L}_{\ast\ast}$. 
By Proposition \ref{prop2} we have 
  	\begin{align*}
  	\psi(u)&= \E_{\nu^{R^\ast}}\left[\frac{e^{R^\ast\cdot R^u_\tau +\kappa_\vT(R^\ast)\cdot \tau}}{\xi(\vT)} \right]\cdot e^{-R^\ast\cdot u}\\
%  	&= \E_{\nu^{R^\ast}}\left[\frac{e^{R^\ast\cdot R^u_\tau +\kappa_%\vT(R^\ast)\cdot \tau}\cdot M_\vT(R^\ast)}{e^{R^\ast\cdot\vT}} 
%\right]\cdot e^{-R^\ast\cdot u}\\
  	&\leq\E_P\bigl[e^{R^\ast\vT}\bigr]\cdot \E_{\nu^{R^\ast}}\left[ e^{R^\ast\cdot R^u_\tau +\kappa_\vT(R^\ast)\cdot \tau}  \right]\cdot e^{-R^\ast\cdot u},
  	\end{align*} 
completing the whole proof.
\end{proof}

It is worth noting that, in the Cram\'{e}r-Lundberg risk model one can construct exponential martingales, and using the stopping theorem one is able to prove   upper bounds for the ruin probabilities. However, this technique does not allow us to prove a lower bound. A method to find also lower bounds for the ruin probabilities  is the ``change of measure technique" for a compound mixed renewal process $S$ developed above.

 \begin{ex}
 \label{bound1}
 \normalfont
 Take $D:=(1,2)$, let $\vT$ be a real-valued random variable on $\vO$, and assume that $P\in\M_{S,{\bf G}(\vT,2)}$, such that $P_{X_1}={\bf Ga}(2,2)$ and $P_{\vT}={\bf U}(1,2)$. Since conditions (a1) and (a2) hold true, it follows by \cite{mt2}, Proposition 3.3, that there exists a $P_\vT$--null set $L_P\in\B((1,2))$ such that $P_\theta\in\M_{S,{\bf Ga}(\theta,2)}$ with $P_{X_1}=(P_\theta)_{X_1}$ for any $\theta\notin L_P$, implying that 
 \[
 \frac{\E_{P_\theta}[X_1]}{\E_{P_\theta}[W_1]}=\frac{1}{\frac{2}{\theta}}=\frac{\theta}{2}\quad\text{ for any }\,\,\theta\notin L_P.
 \]
Put $c(\theta):=\theta+1$ for any $\theta\in D$. As a first step we are going to explicitly determine the function $\kappa_\theta$.  For any $r\in(0,2)$ and $\theta\notin L_{\ast\ast}$ applying condition \eqref{lem35a} and an easy computation we get  
\begin{align*}
&M_{X_1}(r)\cdot (M_{\theta})_{W_1}\bigl(-\kappa_{\theta}(r)-c(\theta)\cdot{r}\bigr)=1 \\
&\Longleftrightarrow  \left(\frac{2}{2-r}\right)^2\cdot\left(\frac{\theta}{\theta+\kappa_{\theta}(r)+c(\theta)\cdot{r}}\right)^2=1\\
 &\Longleftrightarrow  \kappa_{\theta}(r)=\frac{r^2\cdot c(\theta) +r\cdot \theta-2\cdot c(\theta)\cdot r}{2-r}\;\;\mbox{or}\;\;\kappa_{\theta}(r)=\frac{r^2\cdot c(\theta) +r\cdot (\theta-2\cdot c(\theta))-4\theta}{2-r}, 
\end{align*}
equivalently
 \begin{equation} 
\kappa_{\theta}(r)=\frac{r\cdot(r\cdot\theta+r-\theta-2)}{2-r}
\label{equa1}
\end{equation}
or
 \begin{equation} 
\kappa_{\theta}(r)=\frac{r^2(\theta+1)-r(\theta+2)-4\theta}{2-r},
\label{equa2}
\end{equation}
respectively.
By Theorem \ref{clm2} that there exists an adjustment coefficient $R(\theta)\in(0,2)$ with respect to $P_\theta$, for any  $\theta\notin L_{\ast\ast}$, while $R(\theta)$ is the positive solution to the equation $\kappa_{\theta}(r)=0$ by e.g. \cite{scm}, page 133. The latter along with equations \eqref{equa1} and \eqref{equa2} yields $R(\theta)=\frac{\theta+2}{\theta+1}\in(0,2)$ and 
\[
R(\theta)=\frac{\theta+2+(17\theta^2+20\theta+4)^{\frac{1}{2}}}{2(\theta+1)}>2\;\;\mbox{or}\;\; R(\theta)=\frac{\theta+2-(17\theta^2+20\theta+4)^{\frac{1}{2}}}{2(\theta+1)}<0,
\]
respectively; hence $R(\theta)=\frac{\theta+2}{\theta+1}$ is the solution to \eqref{equa1} in $(0,2)$.

 But since $R(\theta)$ is a strictly decreasing function of $\theta$ we get that $R^\ast=\sup_{\theta\in L^c_{\ast\ast}} R(\theta)=\frac{3}{2}\in(0,2)$, implying that $\E_P[e^{R^\ast\cdot X_1}]=\frac{2}{2-\frac{3}{2}}=4<\infty$
 as well as that 
 \[
 \kappa_\theta(R^\ast)=\frac{3}{2}\cdot(\theta-1)\quad\text{ for any }\,\,\theta\notin L_{\ast\ast}.
 \]
 
Put $\ga(x):=R^\ast\cdot x-\ln\E_P[e^{R^\ast\cdot X_1}]$ for any $x\in\vY$. By Theorem \ref{clm2}, for any $\xi\in\mathcal{R}_+(D)$  there exist a unique probability measure $Q^{R^\ast}\in\mathcal{M}_{S,\mathbf{\Lambda}(\rho(\vT))}$ determined by condition ($RRM_{\xi}$), and a rcp $\{Q_\theta^{R^\ast}\}_{\theta\in{D}}$ of $Q^{R^\ast}$ over $Q^{R^\ast}_\vT$ consistent with $\vT$ satisfying conditions $Q^{R^\ast}_\theta\in\mathcal{M}_{S,\mathbf{\Lambda}(\rho(\theta))}$ and ($RRM_\theta$) for any $\theta\notin{L}_{\ast\ast}$, and such that for any $u>0$ the probabilities of ruin $\psi_\theta^{R^\ast}(u)$ and $\psi^{R^\ast}(u)$ with respect to $Q_\theta^{R^\ast}$ and $Q^{R^\ast}$, respectively,  are equal to 1. It then follows by Proposition \ref{prop2} that for any $u>0$ and $\theta\notin L_{\ast\ast}$, the ruin probabilities $\psi(u)$ and $\psi_\theta(u)$ satisfy conditions 
 
 	\[
 	\psi_\theta(u)=\E_{Q_{\theta}^{R^\ast}}\bigl[e^{R^{\ast}r_\tau^u(\theta)+\kappa_\theta(R^\ast)\cdot\tau}\bigr]\cdot{e}^{-R^{\ast}u}=\E_{Q_{\theta}^{R^\ast}}\bigl[e^{\frac{3}{2}\cdot r_\tau^u(\theta)+\frac{3}{2}\cdot(\theta-1)\cdot\tau}\bigr]\cdot{e}^{-\frac{3}{2}\cdot u}
 	\]
 	 and
 	\[
 	\psi(u)= \E_{Q^{R^\ast}}\left[\frac{e^{R^\ast\cdot R^u_\tau +\kappa_\vT(R^\ast)\cdot \tau}}{\xi(\vT)} \right]\cdot e^{-R^\ast\cdot u}=\E_{Q^{R^\ast}}\left[\frac{e^{\frac{3}{2}\cdot R_\tau^u +\frac{3}{2}\cdot(\vT-1)\cdot\tau}}{\xi(\vT)} \right]\cdot e^{-\frac{3}{2}\cdot u}
 	\]
  \end{ex}

 \begin{ex}
	\label{bound2}
 \normalfont
	Take $D:=\vY$, let $\vT$ be a real-valued random variable on $\vO$, and assume that $P\in\M_{S,{\bf G}(\vT,2)}$, such that $P_{X_1}={\bf Ga}(2,2)$ and $P_{\vT}={\bf Ga}(b,a)$, where $(b,a)\in\vY^2$. Since conditions (a1) and (a2) hold true, it follows by \cite{mt2}, Proposition 3.3, that there exists a $P_\vT$--null set $L_P\in\B(D)$ such that $P_\theta\in\M_{S,{\bf Ga}(\theta,2)}$ with $P_{X_1}=(P_\theta)_{X_1}$ for any $\theta\notin L_P$, implying that 
	\[
	\frac{\E_{P_\theta}[X_1]}{\E_{P_\theta}[W_1]}=\frac{1}{\frac{2}{\theta}}=\frac{\theta}{2}\quad\text{ for any }\,\,\theta\notin L_P.
	\]
	
 Put $c(\theta):=\theta$ for any $\theta\in D$.  For any $r\in(0,2)$ and $\theta\notin L_{\ast\ast}$ we get as in Example \ref{bound1}  that there exists an adjustment coefficient $R(\theta)\in(0,2)$ with respect to $P_\theta$ being the solution to the equation 
\begin{equation}
\kappa_{\theta}(r)=\frac{r\cdot\theta\cdot (r-1)}{2-r}=0
\label{equa3}
\end{equation}
for any $\theta\notin{L}_{\ast\ast}$. Thus, we get 
$R^\ast=\sup_{\theta\in L^c_{\ast\ast}} R(\theta)=1\in(0,2)$, implying that $\E_P[e^{R^\ast\cdot X_1}]=\frac{2}{2-1}=2<\infty$.

Put $\ga(x):=R^\ast\cdot x-\ln\E_P[e^{R^\ast\cdot X_1}]$ for any $x\in\vY$. By Theorem \ref{clm2} for any $\xi\in\mathcal{R}_+(D)$  there exist a unique probability measure $Q^{R^\ast}\in\mathcal{M}_{S,\mathbf{\Lambda}(\rho(\vT))}$ determined by condition ($RRM_{\xi}$), and a rcp $\{Q_\theta^{R^\ast}\}_{\theta\in{D}}$ of $Q^{R^\ast}$ over $Q^{R^\ast}_\vT$ consistent with $\vT$ satisfying conditions $Q^{R^\ast}_\theta\in\mathcal{M}_{S,\mathbf{\Lambda}(\rho(\theta))}$ and ($RRM_\theta$) for any $\theta\notin{L}_{\ast\ast}$, and such that for any $u>0$ the probabilities of ruin $\psi_\theta^{R^\ast}(u)$ and $\psi^{R^\ast}(u)$ with respect to $Q_\theta^{R^\ast}$ and $Q^{R^\ast}$, respectively,  are equal to 1. It then follows by Proposition \ref{prop2} that for any $u>0$ and $\theta\notin L_{\ast\ast}$, the ruin probabilities $\psi(u)$ and $\psi_\theta(u)$ satisfy conditions 
	
	\[
	\psi_\theta(u)=\E_{Q_{\theta}^{R^\ast}}\bigl[e^{R^{\ast}r_\tau^u(\theta)+\kappa_\theta(R^\ast)\cdot\tau}\bigr]\cdot{e}^{-R^{\ast}u}=\E_{Q_{\theta}^{R^\ast}}\bigl[e^{ r_\tau^u(\theta)}\bigr]\cdot{e}^{-u}\leq e^{-u}
	\]
	and
	\[
	\psi(u)= \E_{Q^{R^\ast}}\left[\frac{e^{R^\ast\cdot R^u_\tau +\kappa_\vT(R^\ast)\cdot \tau}}{\xi(\vT)} \right]\cdot e^{-R^\ast\cdot u}=\E_{Q^{R^\ast}}\left[\frac{e^{R_\tau^u }}{\xi(\vT)} \right]\cdot e^{-u}\leq e^{- u},
	\]
where the inequalities follow by Theorem \ref{clm2}.	
\end{ex}

\end{document}